\documentclass{amsart}

\usepackage[latin1]{inputenc}
\usepackage{amssymb}

\usepackage[english]{babel}

\newtheorem{theorem}{Theorem}[section]
\newtheorem{lem}[theorem]{Lemma}
\newtheorem{e-proposition}[theorem]{Proposition}
\newtheorem{corollary}[theorem]{Corollary}
\newtheorem{e-definition}[theorem]{Definition\rm}

\begin{document}


\selectlanguage{english}
\title[Composition operators]{Composition operators from logarithmic Bloch spaces to weighted Bloch spaces}


\selectlanguage{english}

\author{R. E. Castillo}

\address{Ren\'{e} E. Castillo\\Departamento de Matem\'aticas, Universidad Nacional de Colombia\\AP360354 Bogot\'{a}, Col\'ombia }

\email{recastillo@unal.edu.co}

\author{D. D. Clahane$^*$}

\address{Dana D. Clahane\\Mathematics $\&$ Computer Science Division, Fullerton
College\\321 East Chapman\\Fullerton, California,
USA\\(714)-992-7390}

\address{{\em Current Address:}\\Mathematics Department\\McCarthy Hall 154\\California
State University\\Fullerton, CA 92834}

\email{dclahane@fullcoll.edu}

\author{J. F. Far\'{i}as López}

\address{Juan F. Far\'{i}as Lopez\\Julio C. Ramos Fern\'{a}ndez\\Departamento de Matem\'atica, Universidad de Oriente\\ 6101 Cuman\'a,
Edo. Sucre, Venezuela}

\email{juan.farias.lopez@gmail.com}

\author{J. C. Ramos Fern\'andez}
\thanks{$^*$: Corresponding author.\\ \indent The second author is partially supported by a Research Sabbatical at Fullerton College, and his research
interaction with students is partially supported by ENGAGE in STEM
(Encouraging New Graduates and Gaining Expertise in Science,
Technology, Engineering and Math), a Special Program funded by the
United States Department of Education Hispanic Serving Institutions
(HSI) STEM and Articulation Programs cooperative arrangement grant
project involving California State University, Fullerton, Fullerton
College, and Santa Ana College.\\\indent The fourth author  has been
partially supported by the Comisi\'on de Investigaci\'on,
Universidad de Oriente, Project CI-2-010301-1678-10.}
\email{jcramos@udo.edu.ve}


\begin{abstract}
We characterize the analytic self-maps $\phi$ of the unit disk
${\Bbb D}$ in ${\Bbb C}$ that induce continuous composition
o\-pe\-ra\-tors $C_\phi$ from the log-Bloch space
$\mathcal{B}^{\log}({\Bbb D})$ to $\mu$-Bloch spaces ${\mathcal
B}^\mu({\Bbb D})$ in terms of the sequence of quotients of the
$\mu$-Bloch semi-norm of the $n$th power of $\phi$ and the log-Bloch
semi-norm (norm) of the $n$th power $F_n$ of the identity function
on ${\Bbb D}$, where $\mu:{\Bbb D}\rightarrow (0,\infty)$
 is continuous and bounded.  We also obtain an expression that is
equivalent to the essential norm of $C_\phi$ between these spaces,
thus characterizing $\phi$ such that $C_\phi$ is compact.  After
finding a pairwise norm equivalent family of log-Bloch type spaces
that are defined on the unit ball ${\Bbb B}_n$ of ${\Bbb C}^n$ and
include the log-Bloch space, we obtain an extension of our
boundedness/compactness/essential norm results for $C_\phi$ acting on ${\mathcal B}^{\log}$ to the case when $C_\phi$ acts on these more general log-Bloch-type spaces.

\vspace{0.1cm}


\end{abstract}

\maketitle


\section{Introduction}
\subsection{Domains considered and weighted Bloch spaces}
Let $\, {\Bbb D}$ denote the unit disk in the complex plane $\Bbb
C$, and denote by $\,{\mathcal H}({\Bbb D})\, $ the linear space of
all holomorphic functions on ${\Bbb D}$.  Throughout this paper,
$\log$ denotes the natural logarithm function, and $\mu$ denotes
what we call a {\em weight} on ${\Bbb D}$; that is, $\mu$ is a
bounded, continuous and strictly positive function defined on ${\Bbb
D}$. The {\em $\mu$-Bloch space} $\, \mathcal{B}^\mu({\Bbb D})$,
which we denote more briefly by ${\mathcal B}^{\mu}$, consists of all
$f\in{\mathcal H}({\Bbb D})$ such that
$$
 \left\|f\right\|_{\mu}:=\sup_{z\in{\Bbb
 D}}\mu\left(z\right)\left|f'(z)\right|<\infty.
$$
$\mu$-Bloch spaces are called {\em weighted Bloch spaces}. When
$\mu(z)=1-|z|^2$, ${\mathcal B}^{\mu}$ becomes the classical Bloch
space $\mathcal{B}({\Bbb D})$. If $\alpha\geq 0$ and $\mu:{\Bbb
D}\to(0,1)$ is given by $\mu(z)=\left(1-|z|^2\right)^\alpha$, then
we denote $||\cdot||_\mu$ by $||\cdot||_\alpha$, and in this case,
${\mathcal B}^\mu({\Bbb D})$ is denoted by ${\mathcal
B}^\alpha({\Bbb D})$, the so-called {\em $\alpha$-Bloch space} of
${\Bbb D}$. For weights $\mu$ on ${\Bbb D}$, a Banach space
structure (cf. \cite{StevicAgarwal}) on ${\mathcal B}^{\mu}({\Bbb
D})$ arises if it is given the norm
$$\|f\|_{\mathcal{B}^\mu}:=\left|f(0)\right|+||f||_\mu.$$
These Banach spaces provide a natural setting in which one can study
properties of various operators. For instance, K. Attele in \cite{At92}
proved that if $\mu_1(z):=w(z)\log\frac{2}{w(z)}$, where
$w(z):=1-|z|^2$ and $z\in\Bbb D$, then the Hankel operator $H_f$
induced by a function $f$ in the Bergman space $A^2({\Bbb D})$ (see
\cite[~Ch.~2]{CowMac95}) is bounded if and only if $\,f\in
B^{\mu_1}({\Bbb D})$, thus giving one reason, and not the only
reason, why log-Bloch-type spaces are of interest.

\subsection{Definition of the log-Bloch space}

For notational convenience, we will state and prove our main results
for composition operators acting on the $\mu$-Bloch space ${\mathcal
B}^{v_{\log}}({\Bbb D})$, where $v_{\log}:{\Bbb
D}\rightarrow(0,\infty)$ is the weight given by
\begin{equation}\label{def-peso}
 v_{\log}(z)=\left(1-|z|\right)\log\left(\frac{3}{1-|z|}\right).
 \end{equation}

We will also denote ${\mathcal B}^{v_{\log}}({\Bbb D})$  by
${\mathcal B}^{\log}$ and the semi-norm $||f||_{{\mathcal
B}^{\log}}-|f(0)|$ by $||f||_{\text{log}}$ rather than
$||f||_{v_{\log}}$.

\subsection{Definition of a composition operator}

During the past decade, there has been a surge in new results
concerning various linear operators $L : X\to  Y$ where at least one
of the spaces $X$ and $Y$ is a space of functions satisfying a
Bloch-type growth condition.  A steadily increasing amount of
attention has been paid to the case when $L=C_\phi$, a so-called
{\em composition operator}, which we now define.

\vspace{0.2cm} Let $ \mathcal{H}_1 $ and $ \mathcal{H}_2$ be two
linear subspaces of $\, {\mathcal H}({\Bbb D})\, $. If $ \phi \, $
is a holomorphic self-map of $ \, {\Bbb D},\, $ such that $ \, f
\circ \phi \, $ belongs to $ \, \mathcal{H}_2 \, $ for all $ \, f
\in \mathcal{H}_1, \, $ then $ \, \phi \, $ induces a linear
operator $ \, C_{\phi} \, : \mathcal{H}_1 \, \to \, \mathcal{H}_2 \,
$ defined by
$$
 C_{\phi}(f) : = f  \circ  \phi.
$$
$C_\phi$ is called the \textit{composition operator} with
\textit{symbol} $ \, \phi. \, $  Composition operators continue to
be widely studied on various subspaces of $\, {\mathcal H}({\Bbb
D})\, $.  A standard introductory reference for the theory of
composition operators is the monograph by C. Cowen and B. MacCluer
\cite{CowMac95}, and another useful introduction to composition and
other operators, particularly on Bloch-type spaces, is contained in
the book by K. Zhu \cite{Zhu07}. A lively introduction to
composition operators on analytic function spaces of one complex
variable is given in J. Shapiro's book on the subject
\cite{Shapiro}. We care here about $C_\phi$ as it acts between
$X={\mathcal B}^{\log}$ and $Y={\mathcal B}^\mu$ for an arbitrary
weight $\mu$ on ${\Bbb D}$, and later in the paper, we will consider
$C_\phi$ as it acts on a more general family of spaces that include
${\mathcal B}^{\log}$.

It is natural to consider extensions of the above results, and the
results of the present paper, to a generalization of Bloch-type
spaces, called Bloch-Orlicz spaces, the first of which was
introduced in \cite{Ra10}:  Let $\mu$ be a weight on ${\Bbb D}$, and
let $\Phi:[0,\infty)\rightarrow[0,\infty)$ be a strictly increasing,
convex function such that $\Phi(0)=0$, with
$\lim_{t\rightarrow\infty}\Phi(t)=\infty$.  Then we define the {\em
$(\mu,\Phi)$-Bloch-Orlicz space} ${\mathcal B}^{\mu}_{\Phi}({\Bbb
D})$ to be the collection of all $f\in{\mathcal H}({\Bbb D})$ such
that there is a $\lambda>0$ with
\[
\sup_{z\in {\Bbb D}}\mu(z)\Phi(\lambda|f'(z)|)<\infty.
\]
Note that if $\mu(z):=1-|z|^2$ above, then J. Ramos Fernandez'
``Bloch-Orlicz space" is obtained as he defined it in \cite{Ra10},
and upon which he studied superposition operators jointly with R.
Castillo and M. Salazar in \cite{CRS}. The theory of composition
operators from and/or to ${\mathcal B}_\Phi^\mu$, which could be
developed alongside the parallel, emerging theory of composition
operators from and/or to ``Orlicz-ised" transformations of classical
integrally defined function spaces, say, the Hardy-Orlicz (cf.
\cite{Ro}) and Bergman-Orlicz (cf. \cite{Sharma}) spaces, seems
interesting.

We also point out here that log Bloch-type spaces are not simply
``made-up spaces."  Showing that outer functions on ${\mathcal B}$
are weak*-cyclic, L. Brown and A. L. Shields proved an auxiliary
result in \cite{BrownShields} which says that when $k=2$ and
$\theta=1$, the space ${\mathcal B}^{\log}_{k,\theta}$ that we
define in the last section of this paper coincides with the
multiplier space of ${\mathcal B}$.  In turn, it seems rather
fundamental and natural to study the size of composition from
log-Bloch type spaces to other (more general) weighted Bloch spaces.

\subsection{Some related results concerning composition operators on
Bloch-type spaces}

In \cite{MadMat95}, K. Madigan and A. Matheson characterized the
maps $\phi$ that generate, respectively, continuous and compact
composition operators $C_\phi$ on $\, {\mathcal B}$.  In turn, their
results were extended by Xiao \cite{Xiao} to the $\alpha$-Bloch
spaces ${\mathcal B}^\alpha({\Bbb D})$ for $\alpha>0$ and by Yoneda
\cite{Yo02} to ${\mathcal B}^{\log}$.

After he introduced a more general family of log-Bloch-type spaces that include ${\mathcal B}^{\log}$ in \cite{Stevic}, S. Stevi\'{c} obtained, jointly with R. Agarwal in \cite{StevicAgarwal}, function theoretic characterizations of holomorphic functions $\psi$ and holomorphic self-maps $\phi$ of ${\Bbb D}$ such that the weighted composition operator $W_{\psi,\phi}$ on these spaces defined by $W_{\psi,\phi}(f)=\psi(f\circ\phi)$ is bounded or compact from these spaces to ${\mathcal B}^\mu$ where $\mu$ is a weight.  Also, in \cite{ZhanX05}, X. Zhang and J. Xiao characterized the holomorphic
functions $\psi$ on the complex Euclidean unit ball ${\Bbb B}_n$ of
${\Bbb C}^n$, together with the holomorphic self-maps $\phi$ of
${\Bbb B}_n$, such that $W_{\psi,\phi}$ is bounded or compact between similarly defined $\,
\mu$-Bloch spaces of ${\Bbb B}_n$. For $n>1$, it is required here
that $\, \mu \, $ be a so-called \emph{normal} function. The results
of Zhang and Xiao were extended by H. Chen and P. Gauthier
\cite{ChGa09} to the $\mu$-Bloch spaces of ${\Bbb B}_n$ for which
$\mu$ is a positive and non-decreasing continuous function such that
$\mu(t)\to 0$ as $t\to 0$ and $\mu(t)/t^\delta$ is decreasing for
small $t$ and for some $\delta
>0$.

\vspace{0.2cm} Other compactness criteria for composition operators
on Bloch spaces have been found by M. Tjani \cite{Tj}, and more
recently, H. Wulan, D. Zheng, and K. Zhu \cite{WZheZhu09} proved the following
result:
\begin{theorem}\cite{WZheZhu09}\label{th-WZZ09}
Let $\phi$ be an analytic self-map of $\Bbb D$. Then $C_\phi$ is
compact on ${\mathcal B}$ if and only if
$$
     \lim_{j\to\infty}\|\phi^j\|_{{\mathcal B}}=0.
    $$
   \end{theorem}

In \cite{GiMalR09}, J. Gim\'enez, R. Malav\'e, and J. C. Ramos
Fern\'andez extended the results of \cite{MadMat95} to certain
$\mu$-Bloch spaces, where the weight $\mu$ can be extended to a
non-vanishing, complex-valued holomorphic function that satisfies a
reasonable geometric condition on the Euclidean disk $\, D(1,1)$.
Ramos Fern\'andez in \cite{Ra10} introduced Bloch-Orlicz spaces, to
which he extended all of the results mentioned above.

\subsection{The essential norm of an operator}

The essential norm of a continuous linear operator
$T$ between normed linear spaces $X$ and $Y$ is its distance from
the compact operators; that is, $\|T\|^{X\rightarrow Y}_e =
\inf\left\{\|T - K\| : K\, \text{ is compact}\right\}$, where
$\|\cdot\|$ denotes the operator norm. Notice that
$\|T\|^{X\rightarrow Y}_e = 0$ if and only if $T$ is compact, so
that estimates on $\|T\|_e^{X\rightarrow Y}$ lead to conditions for
$T$ to be compact.

\subsection{Previous results on the essential norm of $C_\phi$ on Bloch-type spaces}

The essential norm of a composition operator on ${\mathcal B}({\Bbb
D})$ was calculated by A. Montes Rodr\'{\i}guez in \cite{MoRo99}.
He obtained similar results for essential norms of weighted composition
operators between weighted Banach spaces of analytic functions in \cite{MoRo00}.  Other results in this direction appear in the paper by M. D. Contreras and A. G. Hern\'andez D\'{\i}az in \cite{ContH00}; in
particular, formulas for the essential norm of  weighted composition
operators on the $\alpha$-Bloch spaces of ${\Bbb B}_n$ were obtained
(see also the paper of B. MacCluer and R. Zhao \cite{MacZhao03}).
Recently, many extensions of the above results have appeared in the
literature; for instance, the reader is referred to the paper of
R. Yang and Z. Zhou \cite{YaZho10} and several references therein. Zhao in
\cite{Zhao10} gave a formula for the essential norm of
$C_\phi:\mathcal{B}^\alpha({\Bbb D})\to\mathcal{B}^\beta({\Bbb D})$
in terms of an expression involving norms of powers of $\phi$. More
precisely, he showed that
 $$
  \left\|C_\phi\right\|_e^
  {\mathcal{B}^\alpha({\Bbb D})\to\mathcal{B}^\beta({\Bbb D})}=\left( \frac{e}{2\alpha} \right) ^\alpha\limsup_{j\to\infty}j^{\alpha-1}\left\|\phi^j\right\|_{{\mathcal B}^{\beta}({\Bbb D})}.
 $$
It follows from the discussion at the beginning of this paragraph
that $C_\phi:\mathcal{B}^\alpha({\Bbb D})\to\mathcal{B}^\beta({\Bbb
D})$ is compact if and only if
 $$
  \lim_{j\to\infty}j^{\alpha-1}\left\|\phi^j\right\|_{{\mathcal B}^{\beta}({\Bbb D})}=0.
 $$

O. Hyv\"{a}rinen, M. Kemppainen, M. Lindstr\"{o}m, A. Rautio, and E.
Saukko in \cite{gblochpaper} recently obtained necessary and
sufficient conditions for boundedness and an expression
characterizing the essential norm of a weighted composition operator
between general weighted Bloch spaces ${\mathcal B}^{\mu}$, under
the technical requirements that $\mu$ is radial, and that it is
non-increasing and tends to zero toward the boundary of ${\Bbb D}$.
The results in the present paper, in contrast, examine a more
concrete domain space for $C_\phi$ but now for the case when the
target space of $C_\phi$ is {\em any} weighted Bloch type space not
necessarily having the aforementioned requirements.

 \subsection{Objectives and organization of the paper}

 The goal of this paper is to extend the results in \cite{WZheZhu09} for composition operators on ${\mathcal B}({\Bbb D})$ and \cite{Zhao10} for
 composition operators
 between $\alpha$-Bloch spaces of ${\Bbb D}$
 to the more general case of composition operators between log Bloch-type spaces (including the so-called log Bloch space as a special case)
 and $\mu$-Bloch spaces.  We will first show that if $\mu$ is a weight on $\Bbb
 D$ and $\phi:{\Bbb D}\rightarrow{\Bbb D}$ is holomorphic, then the following statements hold:
 \begin{itemize}
  \item (Theorem \ref{th-acotado} below) $C_\phi:\mathcal{B}^{\log}\to \mathcal{B}^{\mu}$ is bounded if and only if
    $$
     \sup_{j\in\Bbb N} \frac{\left\|\phi^j\right\|_{{\mu}}}{\left\|F_j\right\|_{\log}}<\infty.
    $$
  \item (Theorem \ref{th-gen-MM} below) $C_\phi:\mathcal{B}^{\log}\to \mathcal{B}^{\mu}$ is compact if and only if
    $$
      \lim_{j\to \infty}\frac{\left\|\phi^j\right\|_{\mu}}{\left\|F_j\right\|_{\log}}=0.
    $$
  \item (Theorem \ref{th-norm-esencial} below) If $C_\phi:\mathcal{B}^{\log}\to \mathcal{B}^{\mu}$, then
  $$
   \left\|C_\phi\right\|_e^{\mathcal{B}^{\log}\to \mathcal{B}^{\mu}}\sim
   \limsup_{j\to\infty}\frac{\left\|\phi^j\right\|_{{\mu}}}{\left\|F_j\right\|_{\log}}.
  $$
  \end{itemize}
  Above, and in what follows, $(F_j)_{j\in{\Bbb N}}$ denotes the sequence of monomial functions on ${\Bbb D}$ given by $F_j(z)=z^j$,
  and, for two positive variable quantities $A$ and $B$, we write $A\sim B$
and say that $A$ {\em is equivalent to }$B$ if and only if there is
a positive constant $K$ such that $\frac{1}{K}A\leq B\leq K\, A$ for
all values of $A$ and $B$.

In the section that follows, we will obtain results that are
auxiliary to proving Theorem \ref{th-acotado}, which we state and
prove in Section 3.  After stating and proving some auxiliary
compactness results in Section 4, we prove Theorems
\ref{th-norm-esencial} and \ref{th-gen-MM} in Section 5.  In Section
6, we will point out that ${\mathcal B}^{\log}$, which has been
called the ``log-Bloch space", is a member of a two-parameter,
pairwise norm-equivalent family of log-Bloch-type spaces that we
call $(k,\theta)$-log Bloch spaces, this particular log-Bloch space
corresponding to $k=1$ and $\theta=3$. In Section 7, we will
immediately thus obtain that all of our results for composition
operators hold more generally when $C_\phi$ acts on the
$(k,\theta)$-log Bloch spaces for $k=1$ or $2$ and $\theta\geq e$.

\section{Auxiliary Facts}

\medskip

Throughout the rest of this paper, we let $L=1-1/\log 3$.  The
following lemma, which we prove for the sake of completeness, will
only be needed to prove the one that follows it:
\begin{lem}\label{elem}
Define $A:[1,\infty)\rightarrow(0,1]$ by
\[
A(x)=\left(\frac{x}{x+L}\right)^{x-1}.
\]
Then we have that
\[
\inf_{x\geq 1}A(x)=\lim_{x\rightarrow\infty}A(x)=e^{-L}.
\]
\end{lem}
{\it Proof:} We have that
\[
\lim_{x\rightarrow\infty}\log[A(x)]=\lim_{x\rightarrow\infty}-L\frac{x-1}{x+L}\log\left(1+\frac{-L}{x+L}\right)^{\frac{x+L}{-L}}=-L.
\]
The second equation in the statement of the lemma immediately
follows.  The first equation is obtained by observing that
\[
\frac{d}{dx}[\log A(x)]=\log \frac{x}{x+L}+\frac{Lx-L}{x(x+L)} <\log
\frac{x}{x+L}+\frac{L}{x+L}=\log \frac{x}{x+L}-\frac{x}{x+L}+1,
\]
which is negative, since $\log \eta- \eta+1$ takes on values
$-\infty$ and $0$ at $\eta=0$ and $\eta =1$ respectively and is
strictly increasing in $\eta\in(0,1)$. It follows that the
expression $\log A(x)$ (and, in turn, $A(x)$) is strictly decreasing
on $[1,\infty)$. The first equation in the statement of the lemma
then follows from this fact and the second equation in the statement
of the lemma.
 \hfill${\blacksquare}$

\medskip

In what follows, we will have need for a sequence $(r_j)_{j\in{\Bbb
W}}$, where ${\Bbb W}$ denotes the whole numbers.  We define
\begin{equation}\label{rks}
r_0=0\,\,\,{\text{ and }}\,\,\,r_j=1-\frac{L}{j+L}\,\,\,{\text{ for each }}j\in{\Bbb N}.
\end{equation}
The sequence $(r_j)_{j\in{\Bbb W}}$, lies in $[0,1)$, is strictly
  increasing and satisfies $r_j\to 1^-$
as $j\to\infty$.  If $\phi:{\Bbb D}\rightarrow{\Bbb D}$ and
$j\in\Bbb N$, then we define
\[
A^\phi_j:=\{z\in{\Bbb D}:r_{j-1}\leq |\phi(z)|<r_j\}.
\]
Also needed in the proofs of both of our results on composition
operators is the following sequence of functions $(h_j)_{j\in{\Bbb
N}}:(0,1)\rightarrow(0,\infty)$ given by
  $$h_j(t)=\frac{j}{\log(j+1)}t^{j-1}\left(1-t\right)\log\left(\frac{3}{1-t}\right),$$
and the following fact about these functions:

\begin{lem}\label{le-propied-hk}
    Let $j\in{\Bbb N}$.  Then $h_j$ is decreasing on  $[r_{j-1},1)$. Also, we have that
 \begin{equation}\label{cota-inf-z}
  h_j(t)\geq \frac{L}{2e^L} \,\,\,{\text{ for all }}t\in
    [r_{j-1},r_j].
 \end{equation}
\end{lem}

  \noindent{\it Proof. } It suffices to show that $h_j$ is decreasing on $[r_{j-1},r_j]$ and bounded below by the right side of the above inequality
  on the same interval, for each $j\in{\Bbb N}$.  Let $j\in{\Bbb N}$.  Differentiating, we obtain that for all
  $t\in(0,1)$,
  $$
   h_j'(t)=\frac{j}{\log(j+1)}t^{j-2}\left[\left(j-1-jt\right)\log\left(\frac{3}{1-t}\right)+t\right].
  $$
 Since $L\in (0,1)$, we have that for all $t\in [r_{j-1},1)$ the inequality $j-1-jt <0$ holds.  This fact, the above formula for $h_j'(t)$, and the fact that
 $\log\left(\frac{3}{1-t}\right)\geq \log(3)$ holds for all $t\in(0,1)$, together imply that for all $t\in [r_{j-1},1)$,
  $$
    h_j'(t)\leq\frac{j}{\log(j+1)}t^{j-2}\left[\left(j-1-jt\right)\log\left(3\right)+t\right]<0,
  $$
which in turn implies that $h_j$ is decreasing on
$\left[r_{j-1},1\right)$. The first statement in the lemma then
immediately follows from this statement, and the fact that
$(r_j)_{j\in{\Bbb W}}$ is increasing toward $1$ with values in
$[0,1)$.

By the first statement in the lemma, we have that for any $j\in{\Bbb
N}$ and all $t\in\left[r_{j-1},r_j\right]$,
\begin{eqnarray}
    h_j(t) & \geq & h_j\left(r_j\right)\nonumber\\
    &=&\frac{\log\left[\frac{3}{L}(j+L)\right]}{\log(j+1)}\left[\frac{j}{j+L}\right]^{j-1}\left[\frac{jL}{j+L}\right]\nonumber\\
     & > & 1\left(\frac{j}{j+L}\right)^{j-1}\frac{L}{2}\nonumber\\
     & \geq   & e^{-L}\frac{L}{2}=\frac{L}{2e^L}.\nonumber
\end{eqnarray}
The inequality above is due to Lemma \ref{elem}.  The lemma's statement (\ref{cota-inf-z}) follows.
 \hfill$\blacksquare$

\medskip

The corollary that proceeds the following lemma, and not the
lemma itself, will be used throughout this paper; nevertheless, the
lemma and its proof may be of interest. These results require defining, for all $j\in{\Bbb N}$, $H_j:(0,1)\rightarrow{\Bbb R}$, by
\[H_j(t)=\log(j+1)h_j(t).\]
To avoid the appearance of complex fractions in equations involving
limits in the statements of some of our results, here and throughout
the rest of this paper, we say that a real sequence
$(a_j)_{j\in{\Bbb N}}$ {\em is asymptotic to} another real sequence
of non-zero real numbers $(b_j)_{j\in{\Bbb N}}$ and write
``$a_j\simeq b_n$ as $j\rightarrow\infty$" if and only if
\[
\lim_{j\rightarrow\infty}\frac{a_j}{b_j}=1.
\]
\begin{lem}\label{lemy}

 The following statements hold:

 (A) For all such $j$ with $j\geq 11$, there is a unique $t_j\in(0,1)$ such that $H_j(t_j)$ is the absolute maximum of
 $H_j$.

 (B) The sequence $(t_j)_{j\in\{11,12,13,\ldots\}}$ satisfies the
 following three relations:
 \begin{equation}\label{tgoesto1minus}
\lim_{j\rightarrow\infty}t_j=1^-,
 \end{equation}
 where ``$-$" here denotes that $t_j$ tends to $1$
 from the left,
 \begin{equation}\label{expans-asint}
  t_j\simeq 1-\frac{1}{j}+\frac{1}{j\,\log\left(3j\right)}{\text{ as
  }}j\rightarrow\infty,{\text{ and }}
 \end{equation}
 \begin{equation}\label{nicetrick}
  \lim_{j\rightarrow\infty}[j(1-t_j)]=1.
 \end{equation}

 (C) We have that
 $$
  \max_{0<t<1}H_j(t)\simeq \frac{1}{e}\log\left(j+1\right),\hspace{0.2cm} \text{as}\hspace{0.2cm}j\to\infty.
 $$
\end{lem}

\medskip

  \noindent{\it Proof. }
(A) Let $j\in\{11,12,13,\ldots\}$.  We define
$g_j:(0,1)\rightarrow{\Bbb R}$ by
  $$g_j(t)=\left\{(j-1)-jt\right\}\log\left(\frac{3}{1-t}\right)+t,$$
which implies the following estimates, the strict inequality below
being due to the assumption that $j\geq 11$:
  $$
   g_j'(t)=j+1-\frac{1}{1-t}-j\log\left(\frac{3}{1-t}\right)\leq j+1-\log(3)\,j <
   0.
  $$
 Therefore, $g_j$ is strictly decreasing, and since we also have
 that
 \[
 \lim_{t\rightarrow 0^+}g_j(t)=(j-1)\log(3)>
0{\text{ and }}\lim_{t\rightarrow 1^-}g_j(t)=-\infty,
\]
there is a unique  $t_j\in (0,1)$ such that
\begin{equation}\label{g0}
g_j(t_j)=0.
\end{equation}
These facts establish that $g_j(t)>0$ whenever $t\in(0,t_j)$ and
that $g_j(t)<0$ whenever $t\in(t_j,1)$. This statement, and a direct
calculation that
\[
H_j'(t)=jt^{j-2}g_j\left(t\right){\text{ for all }}t\in(0,1),
\]
together imply that $H_j$ has a unique absolute maximum at $t_j$, as
claimed.

\medskip

(B) Again, let $j\in\{11,12,13,\ldots\}$.  Subsequent addition of
$t_j$ and division by $j^2$ on both sides of Equation (\ref{g0}) gives
that
  \begin{equation}\label{ig-sn}
   \left(\frac{j-1}{j}-t_j\right)\log\left(\frac{3}{1-t_j}\right)=-\frac{t_j}{j}.
  \end{equation}
Since $t_j\in (0,1)$ for all  $j\in\{11,12,13,\ldots\}$, the right
side of Equation (\ref{ig-sn}) above has limit $0$ as
$j\rightarrow\infty$, and since the second factor of the left side
of this equation is bounded below by $\log 3$, we deduce that
$$
 \lim_{j\rightarrow\infty}\left(t_j-\frac{j-1}{j}\right)=0.
$$
Equation (\ref{tgoesto1minus}) then follows from the triangle
inequality, and in turn, subsequent applications of Equation
(\ref{tgoesto1minus}), and the fact that the right side of Relation
(\ref{expans-asint}) tends to $1$ as $j\rightarrow\infty$, verify
that Relation (\ref{expans-asint}) holds.

To prove Equation (\ref{nicetrick}), which we will use to prove Part (C), we employ a technique that is inspired by methods for
solving singularly perturbed non-linear equations (cf.
\cite{SMan98}) as follows: We again use the fact that the second
factor in the left side of Equation (\ref{ig-sn}) is bounded below
by $\log 3$ to obtain that
\[
   \frac{j-1}{j}-t_j=-\frac{t_j}{j\log\frac{3}{1-t_j}},
   \]
which is equivalent to the relation
\[
j-1-jt_j=-\frac{t_j}{\log\frac{3}{1-t_j}}.
\]
It follows that
\[
j-jt_j=1-\frac{t_j}{\log 3-\ln(1-t_j)}.
\]
By Equation (\ref{tgoesto1minus}), the right side of the above
equation has limit $1$ as $j\rightarrow\infty$, and Equation
(\ref{nicetrick}) is now verified.  Thus the proof of Part (B) of
the lemma is now complete.

(C) We first point out that
\[
j\log
t_j=j\frac{\log(1+[t_j-1])}{t_j-1}(t_j-1)=j(t_j-1)\frac{\log(1+[t_j-1])}{t_j-1}.
\]
By Equation (\ref{nicetrick}) and a combination of Equation
(\ref{tgoesto1minus}) and a L'Hopital's rule manipulation of the
fractional expression above, it follows that the left side of the
above equation must tend to $-1(1)=-1$.  It follows from this fact
and another application of Equation (\ref{tgoesto1minus}) that
$(j-1)\log t_j\to -1$ as $j\rightarrow\infty$.  Therefore, we have
that
\begin{equation}\label{tn-1powers}
\lim_{j\rightarrow\infty}t^{j-1}_j=\frac{1}{e}.
\end{equation}
Furthermore, Equation (\ref{nicetrick}) also allows us to observe that
\begin{equation}\label{hmmm}
 \frac{\log\left(\frac{3}{1-t_j}\right)}{\log j}=\frac{\log(3j)}{\log j}-\frac{\log\left(j\left(1-t_j\right)\right)}{\log j}\to
 1{\text{ as }}j\to\infty.
\end{equation}
Part (C) of the lemma then follows from the following chain of equations, the first of Equations (\ref{hover}) below
following from, respectively, Equations (\ref{hmmm}),
(\ref{nicetrick}), and (\ref{tn-1powers}):
\begin{eqnarray}
\lim_{j\to\infty}\frac{e}{\log j}\max_{0<t<1}H_j(t)&=&\lim_{j\to\infty}\frac{e}{\log j}\left(1-t_j\right)\log\left(\frac{3}{1-t_j}\right)t_j^{j-1}\nonumber\\
 &=&e\lim_{j\to\infty}\frac{\log\left(\frac{3}{1-t_j}\right)}{\log j}j\left(1-t_j\right)t_j^{j-1}\nonumber\\
 &=&e(1)(1)\frac{1}{e}=1.\nonumber\hfill\,\,\,\,\,\,\,\,\,\,\,\,\,\,\,\,\,\,\,\,\,\,\,\,\,\,\,\,\,\,\,\,\,\,\,\,\,\,\,\,\,\,\,\,\,\,\,\,\,\,\,\,\,\,\,\,\,\,\,\,\,\,\,\,\,\,\,\,\,\,\,\,\,\,\,\,\,\,\,\,\,\,\,\,\,\,\,\,\,\,\,\,\,\,\,\,\,\,\,\,\blacksquare\label{hover}
\end{eqnarray}

We will have use for the following immediate consequence of Lemma
\ref{lemy}:

\medskip

\begin{corollary}\label{cor-norma-zn}
 $$
  \left\|F_j\right\|_{\log}\simeq\frac{\log(j+1)}{e}{\text{ as
  }}j\to\infty.
 $$
\end{corollary}

\section{Continuity of composition operators from ${\mathcal
B}^{\log}$ to ${\mathcal B}^\mu$}
   In this section, we obtain the following norm growth-rate characterization of the holomorphic self-maps $\phi$ of ${\Bbb D}$ for which
   $C_\phi:{\mathcal B}^{{\text{log}}}\to{\mathcal B}^{\mu}$ is
   continuous, where
 $\mu$ is a fixed weight on ${\Bbb D}$:
\begin{theorem}\label{th-acotado}
    Suppose that $C_\phi:\mathcal{B}^{\log}\to\mathcal{B}^\mu$.  Then $C_\phi$ is bounded if and
only if
    \begin{equation}\label{cond-cont}
    \sup_{j\in\Bbb N} \frac{\left\|\phi^j\right\|_{{\mu}}}{\left\|F_j\right\|_{\log}}<\infty.
    \end{equation}
   \end{theorem}

  \noindent{\it Proof.} $\Rightarrow:$ Suppose first that $C_\phi:\mathcal{B}^{\log}\to\mathcal{B}^\mu$ and that $C_\phi$ is bounded.  Then
there is an $M\geq 0$ such that
    $\left\|C_\phi\, f\right\|_{\mathcal{B}^\mu}\leq M\|f\|_{\mathcal{B}^{\log}}$ for all functions $f\in\mathcal{B}^{\log}$. This statement and the
    facts that $C_\phi(F_j)=\phi^j$,
    $F_j\in \mathcal{B}^{\log}$ and $F_j(0)=0$ for all $j\in{\Bbb N}$ together imply that for all of these $j$'s, we have
    $$
      \frac{\left\|\phi^j\right\|_{\mu}}{\left\|F_j\right\|_{\log}}=\frac{\left\|C_\phi(F_j)\right\|_{\mu}}{\left\|F_j\right\|_{\log}}\leq
\frac{\left\|C_\phi(F_j)\right\|_{{\mathcal
B}^\mu}}{\left\|F_j\right\|_{\log}} \leq
\frac{M\left\|F_j\right\|_{{\mathcal
B}^{\log}}}{\left\|F_j\right\|_{\log}} =
\frac{M\left\|F_j\right\|_{\log}}{\left\|F_j\right\|_{\log}}=M.
    $$
    Inequality (\ref{cond-cont}) immediately follows.

        $\Leftarrow$: Suppose now that Inequality (\ref{cond-cont}) holds.  To show that $C_\phi$ is bounded, we first show that there is an $\widetilde{L}>0$ such that for all $f\in {\mathcal B}^{\log}$,
  \begin{equation}\label{afirmacion}
   \left\|f\circ \phi\right\|_{\mu}  \leq \widetilde{L}\left\|f\right\|_{\log}.
  \end{equation}
 To prove this statement, let $z\in{\Bbb D}$ be fixed.  Then there is a $j\in{\Bbb N}$ such that $|\phi(z)|\in A^{\phi}_j$.  Note that $j$ here depends
 on $\phi$ and
 $z$.  There are two cases to consider: either $j=1$ or $j\geq 2$.  Suppose first that $j=1$, so that in particular, $|\phi(z)|\leq
 r_1$.  Since $t\log(3/t)$ is increasing and positive in $t$ and we
 have now that $1-r_1\leq 1-|\phi(z)|$,
 $$
  v_{\log}\left[\phi(z)\right]\geq
  (1-r_1)\log\frac{3}{1-r_1}=\frac{L}{1+L}\log\frac{3}{\frac{L}{1+L}}=\frac{L}{1+L}\log\frac{3(1+L)}{L}>0.
 $$
Note that the rightmost nonzero quantity above, which we now denote
more briefly by $L_1$, is a constant that depends neither on $z$ nor
$\phi$.  It follows that
\begin{eqnarray}\label{case-0}
  \mu(z)\left|f'\left[\phi(z)\right]\right| \left|\phi'(z)\right|
  &=&  \frac{1}{v_{\log}\left[\phi(z)\right]}\mu(z)\left|\phi'(z)\right|
  v_{\log}\left[\phi(z)\right]\left|f'\left[\phi(z)\right]\right|\nonumber\\
  &\leq&\frac{1}{L_1}\|\phi\|_{\mu}\|f\|_{\log}.
 \end{eqnarray}
 Now suppose that $j\geq 2$.  Then we have that
 \begin{eqnarray*}
  \mu(z)\left|f'\left[\phi(z)\right]\right| \left|\phi'(z)\right|
  &\leq& \|f\|_{\log}\frac{\mu(z)}{v_{\log}\left[\phi(z)\right]}\left|\phi'(z)\right|\\
  &\leq& \|f\|_{\log}\frac{\|\phi ^{j}
  \|_{\mu}}{\|F_{j}\|_{\log}}\frac{\frac{\|F_{j}\|_{\log}}{\log(j+1)}}{h_{j}\left(\left|\phi(z)\right|\right)}.
 \end{eqnarray*}
 By hypothesis,  Lemma \ref{le-propied-hk} (with $t=\left|\phi(z)\right|\in A^{\phi}_j$), and Corollary \ref{cor-norma-zn}, it follows that there is
 an $L_2>0$, depending on neither $f$ nor $z$ in this case $j\geq 2$, such that
 \begin{equation}\label{case-1}
  \mu(z)\left|f'\left[\phi(z)\right]\right| \left|\phi'(z)\right|\leq L_2 \|f\|_{\log}.
 \end{equation}
 Since the respective estimates (\ref{case-0}) and (\ref{case-1}) hold when $j=1$ and $j\geq 2$, it follows that there is an $\tilde{L}$ such that
 Inequality (\ref{afirmacion}) holds for arbitrary $z\in{\Bbb D}$.

Finally, since each $f\in{\mathcal B}^\mu$ is analytic on
$\Bbb D$, we have that for all such $f$,
\[f[\phi(0)]=f(0)+\int_0^{\phi(0)}f'(s)ds=f(0)+\int_0^{\phi(0)}\frac{1}{v_{\log}(s)}v_{\log}(s)f'(s)ds,
\]
from which it follows that
\begin{eqnarray*}
|f[\phi(0)]|&\leq &\left|f(0)\right|+
\int_0^{\phi(0)}\left|\frac{1}{v_{\log}(s)}\right||v_{\log}(s)||f'(s)||ds|\\
&\leq & \left|f(0)\right|+\int_0^{\phi(0)}\left|\frac{1}{v_{\log}(s)}\right|||f||_{\log}|ds|\\
&=&\left|f(0)\right|+\int_0^{\phi(0)}\left|\frac{1}{v_{\log}(s)}\right||ds|||f||_{\log}.
\end{eqnarray*}
Letting $C$ denote the integral quantity in the above expression,
which is finite, non-negative, and independent of $f$, we conclude
that for all $f\in{\mathcal B}^{\mu}$,
    $$
     \left|f\left[\phi(0)\right]\right|+ \left\|f\circ\phi\right\|_{\mu} \leq \left|f(0)\right|+\left(C +
     \widetilde{L}\right)\|f\|_{\log}\leq Q(|f(0)|+||f||_{\log})=Q||f||_{{\mathcal B}^{\log}},
    $$
where $Q:=\max\{1,C+\tilde{L}\}$. The converse portion of the
theorem follows, thus completing the proof of the theorem.
\hfill$\blacksquare$

\section{Auxiliary results on compactness}
Now that we know which analytic self-maps $\phi$ of ${\Bbb D}$
induce bounded composition operators from the log-Bloch space to the
$\mu$-Bloch space, we now turn to the issue of compactness, which we
handle by studying the essential norm of $C_\phi$. Theorem
\ref{th-norm-esencial}, stated and proved in the section that
follows, gives an expression that is equivalent to the essential
norm of $C_\phi$ in this setting.  We will need to prove two
auxiliary facts first, the first of which is the following lemma
that appears in \cite{Tj} and also in \cite{Tj2}.  In both of these
places, we point out a typographical error (``point evaluation
functionals on $X$" there, as one can see from Relation (16) in
\cite{Tj2}, should be ``point evaluations on $Y$").
\begin{lem}[Tjani]\label{le-Tjani}
 Let $X,Y$ be two Banach spaces of analytic functions on $\Bbb D$. Suppose that
 \begin{enumerate}
  \item Point evaluation functionals on $Y$ are bounded.
  \item The closed unit ball of $X$ is a compact subset of $X$ in the topology of uniform convergence on compact sets.
  \item $T:X\to Y$ is continuous when $X$ and $Y$ are given the topology of uniform convergence on compact sets.
 \end{enumerate}
 Then $T$ is a compact operator if and only if given a bounded sequence $(f_j)_{j\in{\Bbb N}}$ in $X$ such that $f_j\to 0$ uniformly on compact sets,
 $(Tf_j)_{j\in{\Bbb N}}\rightarrow 0$ in $Y$.
\end{lem}

We will use Tjani's Lemma to prove the following fact, whose purpose is to prove the lemma that follows it. Though it is familiar to some readers, we will sketch the details of the proof to maintain completeness.

\begin{lem} Suppose that $\mu$ is a weight on ${\Bbb D}$.  Then the following statements hold:

(A) For each compact $K\subset{\Bbb {\Bbb D}}$, there is a $C_K\geq 0$ such
that for all $f\in {\mathcal B}^\mu$ and $z\in K$, we have that
\[
|f(z)|\leq C_K\left\|f\right\|_{\mathcal{B}^{\mu}}.
\]

(B) Every point evaluation functional on $\mathcal{B}^{\mu}$ is
bounded.

(C) The closed unit ball of ${\mathcal B}^\mu$ is compact in the
topology of uniform convergence on compacta.

(D) If $\gamma$ is a weight on ${\Bbb D}$ and $\phi:{\Bbb D}\rightarrow{\Bbb D}$ is holomorphic, with $C_\phi:{\mathcal B}^\mu\rightarrow {\mathcal B}^\gamma$, then $C_\phi$ is continuous with respect to the compact-open subspace topologies on ${\mathcal B}^\mu$ and ${\mathcal B}^\gamma$.
\end{lem}

\begin{proof}
(A): Suppose that $K\subset{\Bbb D}$ is compact, and let $\mu$ be a
weight on ${\Bbb D}$.  Then there is an $r\in[0,1)$ such that
$|z|\leq r$ for all $z\in K$, and for each $z$, the line segment
$[0,z]\subset{\overline D}(0,r)$, the compact Euclidean disk
centered at $0$ with radius $r$.  Since $\mu$ is a weight, it
follows that there must be a $Q>0$ such that $\mu(s)\geq Q$ for all
$s\in D(0,r)$, and in particular, for all $s$ on the line segment
$[0,z]$, for all $z\in K$. Thus for all $z\in K$ and $s\in[0,z]$, we
have that $1/\mu(s)\leq 1/Q$. which implies that for all
$f\in{\mathcal B}^\mu$ and $z\in K$,
$$
 \left|f(z)\right|\leq \left|f(0)\right|+\int_0^z\frac{\left\|f\right\|_{\mu}}{\mu(s)}|ds|\leq
 \left(1+\frac{1}{Q}\right)\left\|f\right\|_{\mu}\leq \left(1+\frac{1}{Q}\right)\left\|f\right\|_{{\mathcal B}^{\mu}}.
$$
Thus the proof of (A) is complete.

(B): Part (B) follows immediately from Part (A).

(C): It follows from Part (A) that the unit ball of ${\mathcal B}^\mu$ is uniformly
bounded on compacta.  Therefore, by Montel's theorem (cf.
\cite{Conw78}), any sequence $(f_n)_{n\in{\Bbb N}}$ in this unit
ball must be a normal family, and there is a subsequence
$(f_{n_k})_{k\in{\Bbb N}}$ that must converge uniformly on compacta
to some $f\in{\mathcal H}(\Bbb D)$.  By another of Montel's
Theorems, $f_{n_k}'\rightarrow f'$ uniformly on compacta as well and
pointwise in particular. We then have that for any $z\in{\Bbb D}$,
\[
\mu(z)|f'(z)|=\lim_{k\rightarrow\infty}\mu(z)|f_{n_k}'(z)|\leq 1,
\]
since $||f_{n_k}||_\mu\leq ||f_{n_k}||_{{\mathcal B}^{\mu}}\leq 1$
for each $k\in{\Bbb N}$. Thus $f\in {\mathcal B}_\mu$, and we thus
have shown that the closed unit ball of ${\mathcal B}_\mu$ is
compact in the compact-open topology, as desired.

(D): Let $f_n\rightarrow f$ uniformly in ${\mathcal B}^\mu$ on compacta, so that in turn, $f_n\circ\phi\rightarrow f\circ\phi$ uniformly on compacta as well.  By assumption, $C_\phi(f)\in{\mathcal B}^\gamma$ and $C_\phi(f_n)\in {\mathcal B}^\gamma$ for all $n\in{\Bbb N}$, and the statement of Part (D) follows.
\end{proof}

Combining the above lemma and Lemma \ref{le-Tjani}, we obtain the
following principal auxiliary result of this section:

 \begin{lem}\label{le-Tj}
  Let $\mu_1$ and $\mu_2$ be weights on $\Bbb D$, and suppose that $\phi:{\Bbb D}\to{\Bbb D}$ is
  holomorphic.  Then $C_\phi:\mathcal{B}^{\mu_1}\to
\mathcal{B}^{\mu_2}$ is compact  if and only if given a
bounded sequence
  $(f_j)_{j\in{\Bbb N}}$ in $\mathcal{B}^{\mu_1}$ such that $\, f_j\to 0\, $ uniformly on compact
   subsets of $\, \Bbb D$, then $\left\|C_\phi(f_j)\right\|_{{\mathcal B}^{\mu_2}}\to 0$ as $j\to\infty$.
 \end{lem}

\section{The essential norm and hence compactness of $C_\phi$ from ${\mathcal B}^{\log}$ to ${\mathcal B}^{\mu}$}
\subsection{An expression that is equivalent to the essential norm}

The main theorem of this section is the following equivalence
result, which we prove after stating two preliminary lemmas below it.

   \begin{theorem}\label{th-norm-esencial}
    \begin{equation}\label{norm-essent}
     \left\|C_\phi\right\|_e^{{\mathcal B}^{\log}\rightarrow{\mathcal B}^{\mu}}\sim\limsup_{j\to\infty}
     \frac{\left\|\phi^j\right\|_{\mu}}{\left\|F_j\right\|_{\log}},
    \end{equation}
    if $\phi$ is a holomorphic self-map of
    ${\Bbb D}$ and $\mu$ is a weight on ${\Bbb D}$ such that $C_\phi$ is
bounded between ${\mathcal B}^{\log}$ and ${\mathcal B}^{\mu}$.
   \end{theorem}

\subsection{Two auxiliary facts}

In order to prove the above theorem, we will need the following two
lemmas, which we will prove for the sake of completeness.  For the
first lemma, we will require the following notation: For $r\in
[0,1]$, define the linear {\em dilation operator} $K_r:{\mathcal
H}({\Bbb D})\to\mathcal{H}({\Bbb D})$ by $K_r\, f=f_r$, where $f_r$,
for each $f\in {\mathcal H}({\Bbb D})$, is given by $f_r(z)=f(rz)$.
For more information about this operator, see \cite{GaLiStevic}.

\begin{lem}\label{ksubrfacts}
Let $r\in[0,1]$.  Then the following statements hold:

(A) ${\mathcal B}^{\log}$ is a $K_r$-invariant subspace of
${\mathcal H}({\Bbb D})$; moreoever, we have that

\[\left\|K_r\right\|^{{\mathcal B}^{\log}\rightarrow {\mathcal B}^{\log}}\leq 1.
\]

(B) If $r\not=1$, then $K_r$ is compact on ${\mathcal B}^{\log}$.

\end{lem}

{\it Proof:}

(A): We omit the proof of this part of the lemma, since it can be
obtained by combining \cite[~Thm.~1,~Part~(e)]{Stevic} in the case
$\beta=\alpha=1$ there with Theorem \ref{kthetaeq} in the present
paper.

(B): To prove this part of the Lemma, we can invoke Lemma
\ref{le-Tjani}, provided we can show that for any bounded sequence
$\left\{f_j\right\}$ in $\mathcal{B}^{\log}$ such that
$f_j\to 0$ uniformly on compacta,
    $$
     \left\|K_r\, f_j\right\|_{{\mathcal B}^{\log}}=\left|K_r\,f_j(0)\right|+\sup_{z\in\Bbb D}v_{\log}(z)\left|\left(K_r\,f_j\right)'(z)\right|\to 0
    $$
    as  $j\to\infty$. Indeed, we observe that the first term on the right side of the above equation tends to zero as $j\rightarrow\infty$, since
    $K_r\,f_j(0)=f_j(0)$ and $f_j$ converges to $0$ on compacta in ${\Bbb D}$, by assumption. Furthermore, we have that $f_j'\rightarrow 0$ uniformly on
    compacta
    (cf. \cite[~p.~142-151]{Conw78}).
This fact and boundedness of $v_{\log}$ together imply that
   $$
     \sup_{z\in\Bbb D}v_{\log}(z)\left|\left(K_r\,f_j\right)'(z)\right|=r\sup_{z\in\Bbb D}v_{\log}(z)\left|f_j'(rz)\right|
     \leq r\left\|v_{\log}\right\|_\infty\sup_{z\in\Bbb D}\left|f_j'(rz)\right|\to 0
   $$
   as  $j\to\infty$. The statement in Part (B) of the Lemma follows.

\medskip

\begin{lem}\label{frjto0}
Suppose that $\left(t_j\right)_{j\in{\Bbb W}}$ is an increasing
sequence in $[0,1)$ that converges to $1$, and let $f\in {\mathcal
H}({\Bbb D})$.  Then we have that
$\left(t_j[f']_{t_j}\right)_{j\in{\Bbb W}}$ converges uniformly to
$f'$ on compact subsets of ${\Bbb D}$.
\end{lem}

\medskip

{\it Proof:} For completeness, we supply some of the details of the proof:  Let $G$ be a
   compact subset of  $\Bbb D$, and let $\varepsilon>0$.  Then $G\subset\overline{D}(0,r)$, where ${\overline{D}}(0,r)$ denotes the Euclidean disk with
   center at the origin in
   ${\Bbb C}$ and radius  $r\in [0,1)$.  Since $f\in{\mathcal H}({\Bbb D})$, it follows that $f'$ is uniformly continuous on ${\overline D}(0,r)$.  One
   verifies, therefore, that $|f'-(f')_{t_j}|\rightarrow 0$ uniformly on ${\overline D}(0,r)$.  Furthermore,
   \[|(f')_{t_j}-t_j(f')_{t_j}|=(1-t_j)|f'|_{t_j}\leq (1-t_j)M\rightarrow 0{\text{ as }}j\rightarrow\infty,
    \]
    so the proof of the lemma can be completed by a straightforward put-and-take, followed by an $\varepsilon/2$-argument involving the triangle
    inequality.

\subsection{Proof of the main essential norm result}

We are now prepared to complete the proof of Theorem
\ref{th-norm-esencial}:

\medskip

\noindent{\it Proof}: Suppose that
$\phi:{\Bbb D}\rightarrow{\Bbb
   D}$, and assume that $\phi$ is holomorphic.  If $\phi$ is the zero function, then the statement of the theorem holds trivially.  Therefore, we can assume
   throughout the sequel that $\phi$ is not the zero function.  Since $||\cdot||_{{\mathcal B}^\mu}$ is a norm, it follows that
   $||\phi||_{{\mathcal B}^{\mu}}>0$. Let $\mu$ be a weight on ${\Bbb D}$.  We set
    \[E:=\limsup_{j\to\infty} \frac{\left\|\phi^j\right\|_{\mu}}{\left\|F_j\right\|_{\log}},
    \]
    which is a finite, non-negative real number, by Theorem
    \ref{th-acotado} and the fact that the monomial functions $F_j\in{\mathcal B}^{\log}$ for all $j\in{\Bbb N}$.  Let $
    K:\mathcal{B}^{\log}\to{\mathcal B}^{\mu}$ be linear and also compact, and define the normalized monomial function sequence
    $(f_j)_{j\in{\Bbb N}}$ in
${\mathcal B}^{\log}$ by
    \[f_j:=\frac{F_j}{\left\|F_j\right\|_{\log}}.
\]
We note that
\begin{equation}\label{fjto0}
f_j\rightarrow 0{\text{ uniformly on compacta in }}{\Bbb D} {\text{
as }}j\rightarrow\infty.
\end{equation}
Since the reverse triangle inequality holds for seminorms, we have
that
\begin{eqnarray*}
\left\|C_\phi\,f_j\right\|_{\mu}-\left\|K\,f_j\right\|_{\mu}
&\leq &
\left\|C_\phi\,f_j-K\,f_j\right\|_{\mu}\\
& \leq & \left\|C_\phi\,f_j-K\,f_j\right\|_{{\mathcal B}^{\mu}}\\
&\leq & ||C_\phi-K||\,||f_j||_{{\mathcal B}^{\log}}\\
&=&||C_\phi-K||\,||f_j||_{\log}\\
&=&||C_\phi-K||.
\end{eqnarray*}
Combining this estimate and the equations
\[
\left\|C_\phi\,f_j\right\|_{\mu}=\left\|C_\phi\left(\frac{F_j}{||F_j||_{\log}}\right)\right\|_{\mu}
   =\frac{1}{||F_j||_{\log}}\left\|C_\phi F_j\right\|_{\mu}
=\frac{\left\|\phi^j\right\|_{\mu}}{\left\|F_j\right\|_{\log}},
\]
we obtain the following inequality:
\[
\frac{\left\|\phi^j\right\|_{\mu}}{\left\|F_j\right\|_{\log}}-\left\|K\,f_j\right\|_{\mu}\leq
||C_\phi-K||.
\]
By taking the $\limsup$ of both sides of the above inequality as
$j\to\infty$ and using Relation (\ref{fjto0}) above along with Lemma
\ref{le-Tj}, we can conclude that
\[
\left\|C_\phi-K\right\|\geq E.
\]
Therefore, we have that
\begin{equation}\label{lowerbd4e}
\left\|C_\phi\right\|_e^{{\mathcal B}^{\log}\to{\mathcal
B}^{\mu}}\geq E.
\end{equation}
Inequality (\ref{lowerbd4e}) implies that the proof of the theorem
will be complete if we can show that the left hand side of the
inequality is bounded above by the product of a constant and $E$.
First, we record the following fact for use later in the proof:
Corollary \ref{cor-norma-zn} implies that we can, in particular,
find an $N\in\Bbb N$ such that
   \begin{equation}\label{des-kN}
    \frac{\left\|F_m\right\|_{\log}}{\log(m+1)}<\frac{3}{2\,e}{\text{ for all }}m\in{\Bbb N}{\text{ such that }}m\geq N.
   \end{equation}
As we noted after we defined it in Equation (\ref{rks}), the
sequence $(r_j)_{j\in{\Bbb W}}$ satisfies $r_j\in[0,1)$ for all
$j\in{\Bbb W}$, so Lemma \ref{ksubrfacts}, Part (B) implies that
\begin{equation}\label{krsubjcompact}
{\text{For all }}j\in{\Bbb W}\,\,\,K_{r_j}{\text{ is compact on }}{\mathcal B}^{\log}.
\end{equation}
This fact and
\cite[~p.~178,~Prop.~3.5]{Conw90} together imply that
\begin{equation}\label{cphiksubrsubjcompact}
C_\phi\,K_{r_j}:\mathcal{B}^{\log}\to\mathcal{B}^{\mu}{\text{ is compact for all }}j\in\Bbb W.
\end{equation}
By Lemma \ref{ksubrfacts}, Part (A), we have that
\begin{equation}\label{ksubrsubjunitball}
||K_{r_j}||^{{\mathcal B}^{\log}\rightarrow{\mathcal B}^{\log}}\leq
1{\text{ for all }}j\in{\Bbb W}.
\end{equation}
Also, for all $j\in{\Bbb W}$, we have that
\begin{eqnarray}
         ||C_\phi||_e^{{\mathcal B}^{\log}\to{\mathcal B}^{\mu}({\Bbb
         D})}
         &=&\inf\left\{||C_\phi-K||{\text{ such that }}K:{\mathcal B}^{\log}\to{\mathcal B}^{\mu}{\text{ is compact}}\right\}
         \nonumber\\
         &\leq & ||C_\phi-K_j||\nonumber\\
         &= &\sup_{\{f\in{\mathcal B}^{\log}:||f||_{{\mathcal B}^{\log}}\leq
         1\}}
         \left\|\left(C_\phi-C_\phi\,K_{r_j}\right)f\right\|_{{\mathcal B}^{\mu}}\label{finalyay}.
         \end{eqnarray}
Thus the proof will be complete if we can show that the norm inside
the above supremum is bounded above by the product of a constant
that does not depend on the choice of $f\in{\mathcal B}^{\log}$ and
$E$. We will break up the norm inside the supremum above into three
pieces, each of which we will show is bounded above by a constant
times $E$. Suppose for the moment that $f\in\mathcal{B}^{\log}$ and
that $\left\|f\right\|_{{\mathcal B}^{\log}}\leq 1$. Since $r_j\to
1^-$ as $j\to\infty$ and each $f_j$ is continuous, then for any
$\varepsilon>0$, we can choose $N'\in{\Bbb W}$ such that
\begin{equation}\label{estimate1of3}
{\text{for all }}j\in{\Bbb W}{\text{ such that }}j\geq
N',\,\,\,\left|f\left[\phi(0)\right]-f\left[r_j\phi(0)\right]\right|<E.
\end{equation}
Since $r_j\in[0,1)$ for all $j\in{\Bbb W}$ and
$\left(r_j\right)_{j\in{\Bbb W}}$ is increasing with limit $1$, we
have by Lemma \ref{frjto0} that
$\left(r_j[f']_{r_j}\right)_{j\in{\Bbb W}}$ converges to $0$
uniformly on compacta in ${\Bbb D}$.  Since  $\{w\in{\Bbb D}:
|w|\leq r_l\}$ is compact, $\left(r_j[f']_{r_j}\right)_{j\in{\Bbb
W}}$ converges to $0$ uniformly on $\{w\in{\Bbb D}: |w|\leq r_l\}$
for each $l\in{\Bbb W}$, and $||\phi||_{{\mathcal B}^{\mu}}>0$, then
for all $l\in{\Bbb W}$, we can find an $N_l\in{\Bbb W}$ such that
for all $j\in{\Bbb W}$ with $j\geq N_l$ and all $z\in{\Bbb D}$ such
that $|\phi(z)|\leq r_l$, we have
\[
\left|r_j(f')_{r_j}\left[\phi(z)\right]-f'\left[\phi(z)\right]\right|<
\frac{E}{||\phi||_{{\mathcal B}^{\mu}}}.
\]
Thus for all $l\in{\Bbb W}$ and for all $j\in{\Bbb W}$ such that
$j\geq N_l$, we have that
\[
\sup_{z\in{\Bbb D}:|\phi(z)|\leq
r_l}\mu(z)\left|r_j(f')_{r_j}\left[\phi(z)\right]-f'\left[\phi(z)\right]\right|\left|\phi'(z)\right|<
||\phi||_{\mu}\frac{E}{||\phi||_{{\mathcal B}^{\mu}}}\leq E.
\]
We record the above statement more briefly below for use later:
\begin{equation}\label{gettinglate}
\sup_{z\in{\Bbb D}:|\phi(z)|\leq
r_l}\mu(z)\left|r_j(f')_{r_j}\left[\phi(z)\right]-f'\left[\phi(z)\right]\right|\left|\phi'(z)\right|<E{\text{
for all }}j\in{\Bbb W}{\text{ such that }}j\geq N_l.
\end{equation}
On the other hand, we have that
\begin{equation}\label{gettingthere11}
   {\text{for all }}l,j\in{\Bbb W},\,\,\,\sup_{\left|\phi(z)\right|> r_l}\mu(z)\left|f'\left[\phi(z)\right]-r_jf'\left[r_j\phi(z)\right]\right|
   \left|\phi'(z)\right|\leq s_l(1)+s_l(r_j),
\end{equation}
  where for $\rho\in[0,1]$, the expression $s_l(\rho)$ is given by
  $$
  s_l(\rho)= \sup_{\left|\phi(z)\right|>
  r_l}\mu(z)\left|f'\left[\rho\phi(z)\right]\right|\left|\phi'(z)\right|,
  $$
if this quantity is finite for all $\rho\in[0,1]$.  Indeed, this
quantity is finite for all such $\rho$, as we will now prove.

Assume now that $\rho\in[0,1]$ and estimate $s_l(\rho)$ in this
case. For such values of $\rho$ and all $l\in{\Bbb W}$ such that
$l\geq N$, we deduce that $s_l(\rho)$ is no larger than
\begin{small}
\begin{eqnarray*}
   & &\sup_{z\in \bigcup_{m=l+1}^\infty A_m^{\phi}}\mu(z)\left|f'\left[\rho\phi(z)\right]\right|\left|\phi'(z)\right|\\
   &\leq &\sup_{z\in \bigcup_{m=l}^\infty A_m^{\phi}}\mu(z)\left|f'\left[\rho\phi(z)\right]\right|\left|\phi'(z)\right|\\
   & \leq&\sup_{m\geq l}\sup_{z\in A_m^{\phi}}\mu(z)\left|f'\left[\rho\phi(z)\right]\right|\left|\phi'(z)\right|
    \frac{m\left|\phi(z)\right|^{m-1}\log(m+1)v_{\log}\left[\rho\phi(z)\right]\left\|F_m\right\|_{\log}}
    {m\left|\phi(z)\right|^{m-1}\log(m+1)v_{\log}\left[\rho\phi(z)\right]\left\|F_m\right\|_{\log}}\\
    &<&\frac{3}{2e}\sup_{m\geq l}\sup_{z\in A_m^{\phi}}\mu(z)\left|f'\left[\rho\phi(z)\right]\right|\left|\phi'(z)\right|
    \frac{m\left|\phi(z)\right|^{m-1}\log(m+1)v_{\log}\left[\rho\phi(z)\right]}
    {m\left|\phi(z)\right|^{m-1}v_{\log}\left[\rho\phi(z)\right]\left\|F_m\right\|_{\log}}\\
    &=&\frac{3}{2e}\sup_{m\geq l}\sup_{z\in A_m^{\phi}}\left|f'\left[\rho\phi(z)\right]\right|\mu(z)|m[\phi(z)]^{m-1}\phi'(z)|
    \frac{\log(m+1)v_{\log}\left[\rho\phi(z)\right]}
    {m\left|\phi(z)\right|^{m-1}v_{\log}\left[\rho\phi(z)\right]\left\|F_m\right\|_{\log}}\\
    &=&\frac{3}{2e}\sup_{m\geq l}\sup_{z\in A_m^{\phi}}\left|f'\left[\rho\phi(z)\right]\right|v_{\log}[\rho\phi(z)]\mu(z)|(\phi^m)'(z)|
    \frac{\log(m+1)}
    {m\left|\phi(z)\right|^{m-1}v_{\log}\left[\rho\phi(z)\right]\left\|F_m\right\|_{\log}}\\
    &\leq&\frac{3||f||_{{\mathcal B}^{\log}}}{2e}\sup_{m\geq l}\frac{||\phi^m||_{\mu}}{||F_m||_{\log}}\sup_{z\in A_m^{\phi}}\frac{\log(m+1)}
    {m\left|\phi(z)\right|^{m-1}v_{\log}\left[\rho\phi(z)\right]}\\
     &\leq&\frac{3}{2e}\sup_{m\geq l}\frac{||\phi^m||_{\mu}}{||F_m||_{\log}}\sup_{z\in A_m^{\phi}}
    \frac{\log(m+1)}
    {m\left|\phi(z)\right|^{m-1}v_{\log}\left[\rho\phi(z)\right]}\\
 &\leq&\frac{3}{2e}\sup_{m\geq l}\frac{||\phi^m||_{\mu}}{||F_m||_{\log}}\sup_{z\in A_m^{\phi}}
    \frac{\log(m+1)}
    {m\left|\phi(z)\right|^{m-1}v_{\log}\left[\phi(z)\right]}\\
    &=&\frac{3}{2e}\sup_{m\geq l}\frac{||\phi^m||_{\mu}}{||F_m||_{\log}}\sup_{z\in A_m^{\phi}}
    \frac{1}{h_m[\phi(z)]}\\
    &<&\left(\frac{2e^L}{L}\right)\left(\frac{3}{2e}\right)\sup_{m\geq l}\frac{||\phi^m||_{\mu}}{||F_m||_{\log}}\\
    &=&\frac{3e^{L-1}}{L}\sup_{m\geq l}\frac{||\phi^m||_{\mu}}{||F_m||_{\log}}\\
    &<&\infty.
\end{eqnarray*}
\end{small}
In the above chain of relations, the third inequality follows from
Inequality
  (\ref{des-kN}).  By definition of $||\cdot||_{{\mathcal B}^{\log}}$, the fourth inequality above is obtained, and the fifth inequality holds by the
  assumption that
  $||f||_{{\mathcal B}^{\log}}\leq 1$.  The sixth inequality above follows from the fact that the continuous extension to $[0,1]$ of $\mu_3$ is
  increasing on $[0,1]$, and the seventh inequality is a consequence of Lemma
  \ref{le-propied-hk}.  Hence, $s_l$
  for $l\geq N$ is a well-defined, real-valued function on $[0,1]$, as claimed,
  and we now have one of three estimates that are needed to complete the proof of the
  theorem.

Note in particular by separate estimation of $s_l(1)$ and
$s_l(r_j)$, which are bounded by the second-from-the-bottom quantity
in the above large chain of inequalities, that for sufficiently
large $l$ and any $j\in{\Bbb W}$,
\begin{eqnarray}
s_l(1)+ s_l(r_j)&\leq & \frac{3e^{L-1}}{L}\sup_{m\geq
l}\frac{||\phi^m||_{\mu}}{||F_m||_{\log}}+
\frac{3e^{L-1}}{L}\sup_{m\geq l}\frac{||\phi^m||_{\mu}}{||F_m||_{\log}}\nonumber\\
              &=&\frac{6e^{L-1}}{L}\sup_{m\geq l}\frac{||\phi^m||_{\mu}}{||F_m||_{\log}}. \label{downstairs}
\end{eqnarray}
Now, for all $j,l\in{\Bbb W}$, we have that
$\displaystyle{\left\|\left(C_\phi-C_\phi\,K_{r_j}\right)f\right\|_{{\mathcal
B}^{\mu}}}$
   \begin{eqnarray}
    &=&\left\|\left(f-f_{r_j}\right)\circ\phi\right\|_{{\mathcal B}^{\mu}}\nonumber\\
    &=&\left|f\left[\phi(0)\right]-f\left[r_j\phi(0)\right]\right|
        +\sup_{z\in\Bbb
        D}\mu(z)\left|f'\left[\phi(z)\right]-r_jf'\left[r_j\phi(z)\right]\right|\left|\phi'(z)\right|\nonumber\\
        &\leq & \left|f\left[\phi(0)\right]-f\left[r_j\phi(0)\right]\right|
        +E^l_{j,1}(f)+E^l_{j,2}(f),\label{inahurry}
   \end{eqnarray}
where
\[E^l_{j,1}(f):=\sup_{|\phi(z)|\leq r_l}\mu(z)\left|f'\left[\phi(z)\right]-r_jf'\left[r_j\phi(z)\right]\right|\left|\phi'(z)\right|\]
        and
        \[
        E^l_{j,2}(f):=
        \sup_{|\phi(z)|>r_l}\mu(z)\left|f'\left[\phi(z)\right]-r_jf'\left[r_j\phi(z)\right]\right|\left|\phi'(z)\right|.
\]
Quantity (\ref{inahurry}) above can be rewritten as
\begin{equation}\label{tired}
\left|f\left[\phi(0)\right]-f\left[r_j\phi(0)\right]\right|   +E^l_{j,1}(f)+s_l(1)+s_l(r_j).
\end{equation}
Now let $l\in{\Bbb W}$ satisfy $l\geq N$, and let $j\in{\Bbb W}$ satisfy $j\geq N'_l:=\max\{N_l,N'\}$.

Indeed, by Inequality (\ref{inahurry}), which is bounded above by Quantity (\ref{tired}), we have that
\begin{eqnarray}
\left\|\left(C_\phi-C_\phi\,K_{r_j}\right)f\right\|_{{\mathcal
B}^{\mu}}
&\leq &\left|f\left[\phi(0)\right]-f\left[r_j\phi(0)\right]\right|+E^l_{j,1}(f)+E^l_{j,2}(f)\nonumber\\
&=&\left|f\left[\phi(0)\right]-f\left[r_j\phi(0)\right]\right|+E^l_{j,1}(f)+s_l(1)+s_l(r_j)\nonumber\\
&<&E+E^l_{j,1}(f)+s_l(1)+s_l(r_j)\label{outtahereyet?}\\
&< &E+E+s_l(1)+s_l(r_j)\label{lookinggood}\\
&\leq &E+E+ \frac{6e^{L-1}}{L}\sup_{m\geq l}\frac{||\phi^m||_{\mu}}{||F_m||_{\log}}\label{upstairs}\\
&=&2E+ \frac{6e^{L-1}}{L}\sup_{m\geq
l}\frac{||\phi^m||_{\mu}}{||F_m||_{\log}}.\nonumber
\end{eqnarray}
Inequality (\ref{outtahereyet?}) follows from Statement
(\ref{estimate1of3}), and Inequality (\ref{lookinggood}) is due to
Statement (\ref{gettinglate}).  Inequality (\ref{upstairs}) follows
from Equation (\ref{downstairs}), together with the inequality that
precedes it. Since the chain of inequalities above holds for all
$l\in{\Bbb W}$ such that $l\geq N$ and in turn for the $j$'s in
${\Bbb W}$ such that $j\geq N'_l$, we can conclude, by Equation
(\ref{finalyay}) and the above estimates, that the essential norm of
$C_\phi$ is bounded above by
\[
\frac{2L+6e^{L-1}}{L}E.
\]
This completes the proof of the theorem
\hfill $\blacksquare$

\subsection{A characterization of the symbols generating compact composition operators}

Theorem \ref{th-norm-esencial} and Corollary \ref{cor-norma-zn}
together immediately imply the following characterization of
analytic symbols $\phi$ that generate compact
$C_\phi:\mathcal{B}^{\log}\to \mathcal{B}^\mu$, thus extending
results in \cite{WZheZhu09,Zhao10}.
 \begin{theorem}\label{th-gen-MM}
Let $\mu$ be a weight on ${\Bbb D}$.
 Suppose that $\phi:{\Bbb D}\to{\Bbb D}$ is
  holomorphic.  Then $C_\phi:\mathcal{B}^{\log}\to \mathcal{B}^\mu$ is compact if and only if either of the following equations holds:

\begin{eqnarray*}
  \lim_{j\to \infty}\frac{\left\|\phi^j\right\|_{\mu}}{\log(j+1)}&=&0.\\
    \lim_{j\to \infty}\frac{\left\|\phi^j\right\|_{\mu}}{\left\|F_j\right\|_{\log}}&=&0.
   \end{eqnarray*}

 \end{theorem}
\vspace{0.2cm}

\subsection{Essential norms of $C_\phi$ from the log-Bloch to $\alpha$-Bloch spaces}

The following result establishing essential norm equivalences and characterizing compact $C_\phi$ from the log-Bloch space to $\alpha$-Bloch spaces is a direct consequence of Theorem \ref{th-norm-esencial}:
\begin{corollary} Let  $\phi:\Bbb D\to\Bbb D$ be analytic, and suppose that $\alpha\geq 0$.  Then the
following statements hold.
\begin{enumerate}
 \item  The essential norm of the continuous operator   $C_\phi:\mathcal{B}^{\log}\to\mathcal{B}^\alpha$  satisfies
    \begin{equation*}
     \left\|C_\phi\right\|_e\sim\limsup_{j\to\infty} \frac{\left\|\phi^j\right\|_\alpha}{\left\|F_j\right\|_{\log}}.
    \end{equation*}
  In particular, this operator is compact if and only if
    \begin{equation*}
     \lim_{j\to\infty} \frac{\left\|\phi^j\right\|_\alpha}{\left\|F_j\right\|_{\log}}=0.
    \end{equation*}

 \item The essential norm of the continuous operator  $C_\phi:\mathcal{B}^{\log}\to\mathcal{B}^{\log}$  satisfies
    \begin{equation*}
     \left\|C_\phi\right\|_e\sim\limsup_{j\to\infty} \frac{\left\|\phi^j\right\|_{\log}}{\left\|F_j\right\|_{\log}}.
\end{equation*}
    In particular, this operator is compact if and only if
    \begin{equation*}
     \lim_{j\to\infty} \frac{\left\|\phi^j\right\|_{\log}}{\left\|F_j\right\|_{\log}}=0.
    \end{equation*}
 \end{enumerate}
\end{corollary}

\section{A pairwise norm-equivalent family of generalized log Bloch spaces}

In this section, we observe that our results concerning composition
operators can be extended to a more general family of spaces that
include the log Bloch space. In addition, this general family of
logarithmic Bloch-type spaces can be defined on the unit ball ${\Bbb
B}_n$ of ${\Bbb C}^n$ induced by the Euclidean inner product, and we
will point out that these spaces are pairwise norm-equivalent to
each other and the log Bloch space for all $n\in{\Bbb N}$.  If $\mu$
is a continuous positive function on ${\Bbb B}_n$, then the {\em
$\mu$-Bloch space ${\mathcal B}^\mu({\Bbb B}_n)$} is defined to be
the Banach space of holomorphic functions on $f$ on ${\Bbb B}_n$
such that
\[
b_\mu^f:=\sup_{z\in {\Bbb B}_n}\mu(z)|(\nabla f)(z)|
\]
is finite, the Banach space structure arising from the norm given by
\[
||f||_{{\mathcal B}^{\mu}({\Bbb B}_n)}:=|f(0)|+b_\mu^f,
\]
which we denote more briefly by $||f||_{\mathcal{B}^\mu}$ as in the case $n=1$.

Suppose that $\theta>1$ and that $k=1$ or $2$. Then we define the
{\em $(k,\theta)$-log Bloch space} ${\mathcal
B}^{\log}_{k,\theta}({\Bbb B}_n)$ to be ${\mathcal B}^\mu({\Bbb
B}_n)$, where the weight $\mu:=v_{\theta}^{(k)}$, and
$v_{\theta}^{(k)}$ is in turn given by
$v_{\theta}^{(k)}(z)=(1-|z|^k)\log\frac{\theta}{1-|z|^k}$.  For the
sake of brevity, we denote the norm of $f\in{\mathcal
B}^{\log}_{k,\theta}({\Bbb B}_n)$ here by $||f||_{k,\theta}$.
 Also, consistent with standard notation in the case $n=1$, we adopt the notation ${\mathcal B}^{\log}({\Bbb B}_n):={\mathcal B}^{\log}_{1,3}({\Bbb
B}_n)$ and call this space the {\em log Bloch space of ${\Bbb
B}_n$}.

The goal of this section is to prove a norm equivalence result from
which we will be be able to extend, in the section that follows, the
main results of this paper to the $(k,\theta)$-log Bloch spaces
defined above, with the more stringent condition $\theta\geq e$,
although we will prove the above-mentioned pairwise norm-equivalence
among these spaces on ${\Bbb B}_n$ for all $n\in{\Bbb N}$, not just
on ${\Bbb D}$.  We leave open the question of whether analogues of
our main results on composition operators hold in the case of ${\Bbb
B}_n$; moreover, after our results were obtained, S. Stevi\'{c}
pointed out to us the papers \cite{Listevic}, \cite{Stevic},
\cite{Stevic2}, \cite{StevicAgarwal}, and \cite{Stevic3}, in which
composition operators (and some natural generalizations of them) are
considered on spaces formed by replacing the ``$(1-|z|)$" and the
logarithm in the definition of the norm on ${\mathcal
B}^{\log}_{k,\theta}({\Bbb B}_n)$ in the special case $k=1$ by
various respective powers of these quantities. The present paper can
be viewed as complementing, in some ways, a number of these results
by Stev\'{i}c and his collaborators.

To prove the pairwise norm-equivalence of the
$(k,\theta)$-logarithmic Bloch spaces for suitable $(k,\theta)$, we
first need the following auxiliary fact:

\begin{lem}\label{unsquare}
Suppose that $\mu:(0,1)\rightarrow(0,\infty)$ is increasing, and
assume that for all $t\in(0,1)$, we have that
\begin{equation}\label{halve}
\mu\left(\frac{1}{2} t\right)\geq \frac{1}{2} \mu(t).
\end{equation}
Then $\mu(1-t)\sim\mu(1-t^2)$.
\end{lem}

\medskip

{\it Proof:}  By the following inequality for all $t\in[0,1]$,
\[
\frac{1}{2}(1-t^2)\leq 1-t\leq 1-t^2,
\]
Inequality (\ref{halve}), and the assumption that $\mu$ is
increasing, we have that for these $t$'s,
\[
\frac{1}{2}\mu\left(1-t^2\right)\leq
\mu\left(\frac{1}{2}(1-t^2)\right)\leq\mu(1-t)\leq\mu(1-t^2).
\]
\hfill$\blacksquare$

\medskip

We are now prepared to prove the following norm equivalence result,
which shows that our main results for composition operators on the
log Bloch space extend to the $(k,\theta)$-log Bloch spaces defined
earlier in this section, for suitable $k$'s and $\theta$'s.  We know
of no reference containing the proof, which may be known to some
readers, so we provide a sketch of the details for the convenience
of other readers.  For $\theta>1$, we will make use of
$\mu_\theta:(0,1]\rightarrow[0,\infty)$ defined by $\mu_\theta(t)=t
\log(\theta/t).$
\begin{theorem}\label{kthetaeq}  Let $\alpha,\beta\geq e$.  Then we
have that
\[
\begin{array}{ccccc}
v_\alpha^{(2)}(z)& \sim & v_\beta^{(2)}(z)& \sim & v_\beta^{(1)}(z),\\
 {\mathcal B}^{\log}_{2,\alpha}&=&{\mathcal B}^{\log}_{2,\beta}&=&{\mathcal B}^{\log}_{1,\beta},\\
||f||_{2,\alpha}&\sim &||f||_{2,\beta}&\sim & ||f||_{1,\beta},
\end{array}
\]
as $f$ varies through these coinciding spaces.
\end{theorem}

\medskip

{\it Proof:}  Let $\alpha,\beta\geq e$.  The second and third set of
equivalences above immediately follow from the top set of
equivalences, which we now prove. We can assume with no loss of
generality, that $\alpha\leq\beta$. We prove the leftmost of these
two equivalences first.  The following two inequalities respectively
follow from the facts that (i) $\log$ is increasing and (ii) for all
$\theta\geq e$, $\mu_\theta$ is increasing:
\begin{equation}\label{chain}
v_\alpha^{(2)}(z)\leq v_\beta^{(2)}(z)\leq \log \beta\,\,\,{\text{ for all }}z\in{\Bbb B}_n.
\end{equation}
Since we have by L'Hopital's Rule that $v_\beta^{(2)}(z)/v_\alpha^{(2)}(z)\rightarrow 1$ as $|z|\rightarrow 1^-$, one checks that there is a $\delta>0$ such that for all $z\in{\Bbb
B}_n$ with $\delta<|z|<1$,
\begin{equation}\label{3over2}
v_\beta^{(2)}(z)\leq \frac{3}{2} v_\alpha^{(2)}(z).
\end{equation}
One verifies that for $\theta>1$, $\mu_\theta$ is bounded with a
removable discontinuity on its graph, at the origin, and the
restriction of $\mu_\theta$ to $(1-\delta^2,1)$ for any
$\delta\in(0,1)$, is therefore bounded away from $0$; in particular,
there is an $m_{\alpha,\delta}>0$, depending only on $\delta$ and
$\alpha$, such that $v_\alpha^{(2)}(z)\geq m_{\alpha,\delta}$ for
all $z\in {\overline{\delta{\Bbb B}_n}}$. By the rightmost
inequality in Relation (\ref{chain}), it follows that for these
$z$'s ,
\[
v_\beta^{(2)}(z)\leq \log \beta=\log
\beta\frac{m_{\alpha,\delta}}{m_{\alpha,\delta}}\leq \log
\beta\frac{v_\alpha^{(2)}(z)}{m_{\alpha,\delta}}=\frac{\log
\beta}{m_{\alpha,\delta}}v_\alpha^{(2)}(z).
\]
We then obtain the leftmost equivalence in the top collection of
relations in the conclusion of the theorem from this fact, the
leftmost inequality in Relation (\ref{chain}), and the fact that
Inequality (\ref{3over2}) holds for all $z\in {\Bbb B}_n$ such that
$\delta<|z|<1$.

We now prove the second equivalence in the top relation appearing in
the conclusion of the theorem. One first checks that since
$\beta\geq e$, $\mu_\beta$ is increasing and concave. In particular,
$\mu_\beta$ satisfies the hypotheses of Lemma \ref{unsquare}, which
allows us to deduce that $\mu_\beta(1-t)\sim \mu_\beta(1-t^2)$.  It
follows that the second equivalence in the top relation in the
conclusion of the theorem holds. \hfill${\blacksquare}$

\medskip

\section{Boundedness, Compactness, and Essential Norms of Composition operators from ${\mathcal B}^{\log}_{k,\theta}({\Bbb
D})$ to ${\mathcal B}^{\mu}({\Bbb D})$}

Theorem \ref{kthetaeq} and the theorems presented in Section 3 and 5
together immediately imply the following more general results on
composition operators from log-Bloch type spaces to weighted Bloch
spaces of ${\Bbb D}$.  To date, various members of this family have
been ambiguously given the same name, ``logarithmic Bloch space". We
do not claim to have found the entire collection of $(k,\theta)\in
{\Bbb R}^2$ to which the results of this paper extend.

\begin{theorem}\label{main}
Suppose that $\theta\geq e$, and assume that $k=1$ or $2$.  Suppose
that $\phi:{\Bbb D}\rightarrow{\Bbb D}$ is holomorphic, and assume
that $\mu$ is a weight on ${\Bbb D}$. Then the following statements
hold:

(A) $C_\phi:\mathcal{B}^{\log}_{k,\theta}\to \mathcal{B}^{\mu}$ is
bounded if and only if
    $$
     \sup_{j\in\Bbb N} \frac{\left\|\phi^j\right\|_{{\mu}}}{\left\|F_j\right\|_{\log}}<\infty.
    $$
    if and only if
$$
     \sup_{j\in\Bbb N} \frac{\left\|\phi^j\right\|_{{\mu}}}{\log(j+1)}<\infty.
    $$

(B) $C_\phi:\mathcal{B}^{\log}_{k,\theta}\to \mathcal{B}^{\mu}$ is
compact if and only if
    $$
      \lim_{j\to \infty}\frac{\left\|\phi^j\right\|_\mu}{\left\|F_j\right\|_{\log}}=0.
    $$
    if and only if
    $$
      \lim_{j\to \infty}\frac{\left\|\phi^j\right\|_{\mu}}{\log(j+1)}=0.
    $$

(C) If $C_\phi:\mathcal{B}^{\log}_{k,\theta}\to \mathcal{B}^{\mu}$,
then
  $$
   \left\|C_\phi\right\|_e^{\mathcal{B}^{\log}_{k,\theta}\to \mathcal{B}^{\mu}}\sim
   \limsup_{j\to\infty}\frac{\left\|\phi^j\right\|_{{\mu}}}{\left\|F_j\right\|_{\log}}\sim
   \limsup_{j\to\infty}\frac{\left\|\phi^j\right\|_{{\mu}}}{\log(j+1)}.
  $$
\end{theorem}

\section{Acknowledgments}

The authors are grateful to Stevo Stevi\'{c} for kindly emailing the
papers \cite{Listevic}, \cite{Stevic}, \cite{Stevic2}, \cite{Stevic3},
 and \cite{StevicAgarwal} and also for kindly making
many very helpful comments on the manuscript.








\end{document}